\documentclass{article}[12pt,paper=a4]
\usepackage{amsmath,amssymb,amsthm,tikz,float,dsfont}
\usetikzlibrary{calc,intersections}
\usepackage{enumitem}
\usepackage{fullpage}

\makeatletter
\let\@fnsymbol\@arabic
\makeatother

\newtheorem{definition}{Definition}
\newtheorem{theorem}{Theorem}
\newtheorem{lemma}[theorem]{Lemma}

\newcommand{\infec} {A}

\newcommand{\checked} {Z}

\newcommand{\filter}[1]{\mathcal{F}(#1)}

\newcommand{\error}{\delta}
\newcommand{\indic}{Q}
\newcommand{\tendstoinfty}{\alpha}

\begin{document}

\title{A simple proof of almost percolation on $G(n,p)$\footnote{An earlier version of this paper appeared in the proceedings of the 27th International Conference on Probabilistic, Combinatorial and Asymptotic Methods for the Analysis of Algorithms (AofA 2016).}}

\author{Mihyun ~Kang\thanks{Graz University of Technology, Institute of Discrete Mathematics, Steyrergasse 30, Graz, Austria.
\newline \{kang, makai\}@math.tugraz.at. Supported by Austrian Science Fund (FWF): P26826.} \and 
Tam\'as ~Makai\footnotemark[2]\\
}

\maketitle

\abstract{We consider bootstrap percolation on the binomial random graph $G(n,p)$ with infection threshold $r\in \mathbb{N}$, an infection process which starts from a set of initially infected vertices and in each step every vertex with at least $r$ infected neighbours becomes infected. We improve the results of Janson, \L uczak, Turova, and Valier (2012) by strengthening the probability bounds on the number of infected vertices at the end of the process, using simple arguments based on martingales and giant components.
}

\section{Introduction}

Bootstrap percolation on a graph with infection threshold $r\in \mathbb{N}$ is a deterministic infection process which evolves in rounds. In each round every vertex has exactly one of two possible states: it is either infected or uninfected. We denote the set of initially infected vertices by $\infec(0)$. In each round of the process every uninfected vertex $v$ becomes infected if it has at least $r$ infected neighbours, otherwise it remains uninfected. Once a vertex has become infected, it remains infected forever. The final infected set is denoted by $\infec_f$.

Bootstrap percolation was introduced by Chalupa, Leath, and Reich \cite{bootstrapintr} in the context of
magnetic disordered systems. Since then bootstrap percolation processes and extensions have been used to describe several complex phenomena: from neuronal activity \cite{MR2728841,inhbootstrap} to the dynamics of the Ising model at zero temperature \cite{Fontes02stretchedexponential}.

The dependence of the final number of infected vertices on the set of initially infected vertices have been studied for a variety graphs. These include various deterministic graphs such as trees \cite{MR2248323,MR2430783}, grids \cite{MR968311,MR2888224,MR2546747,MR1921442,MR1961342}, and hypercubes \cite{MR2214907}. Several random graph models have also been considered including the binomial random graph \cite{arXiv:1602.01751,MR3025687,V07}, random graphs with a given degree sequence \cite{MR2595485}, random regular graphs \cite{MR2283230}, the Chung-Lu model \cite{bootpower,arXiv:1402.2815}, preferential attachment graphs \cite{arXiv:1404.4070}, and random geometric graphs on the hyperbolic plane \cite{MR3426518}.

This paper focuses on bootstrap percolation on the binomial random graph $G(n,p)$, a graph with vertex set $[n]:=\{1,\ldots, n\}$ in which every edge appears independently with probability $p=p(n)$. Extending the results of Vallier \cite{V07}, Janson,  \L uczak, Turova, and Vallier \cite{MR3025687} analysed bootstrap percolation on $G(n,p)$, where the set of initially infected vertices $\infec(0)$ is chosen uniformly at random from the vertex sets of size $a\in [n]$. For $r\geq 2$ and $p$ satisfying both $p=\omega(n^{-1})$ and $p= o(n^{-1/r})$, they showed, among other results, that with probability tending to one as $n\rightarrow \infty$ either only a few additional vertices are infected or almost every vertex becomes infected, depending on the number of initially infected vertices. In addition they determined 
the probability of both of these events up to an additive term tending to zero as $n\rightarrow \infty$.

They also considered the complementary regime of $p$. When $p=O(n^{-1})$, with probability tending to one as $n\rightarrow \infty$, due to the large number of small components, the only way infection can spread to almost every vertex is if almost every vertex was infected initially. In the $p=\Theta(n^{-1/r})$ case, with probability tending to one as $n\rightarrow \infty$ every vertex will become infected if the initial set of infected vertices  tends to infinity. In addition when  $p=\omega(n^{-1/r})$ already $r$ initially infected vertices spread the infection to every vertex in the graph with probability tending to one as $n\rightarrow \infty$. Therefore the most interesting range of $p$ is when $p=\omega(n^{-1})$ and $p= o(n^{-1/r})$ and in this paper we will concentrate on this range.

The main contribution of this paper is a simple proof of almost percolation using martingales and the  giant component.
We introduce a martingale in order to determine the number of infected vertices during the early stages of the process. 
Furthermore in the supercritical regime we show that the subgraph spanned by the vertices with $r-1$ infected neighbours grows large enough to contain a giant component. The infection of just one vertex in this giant component leads to every vertex in the component becoming infected and we show that this in fact happens with exponentially high probability. 

A weaker version of the martingale argument was already used in \cite{MR3025687}. However we are unaware of any other proof relying on the appearance of the giant component, making our proof simple and original.

\paragraph{Main Results.}
Throughout the paper we assume that $r\geq 2$ and that both $p=\omega(n^{-1})$ and \linebreak[4] 
$p=o ( n^{-1/r})$ hold. Moreover, any unspecified limits and asymptotics will be as $n\rightarrow \infty$. Throughout the paper we assume that $n$ is sufficiently large. Let 
$$\error:=\max\left\{(np^{r})^{1/(2(r-1))},(np)^{-1/(4(r-1))}\right\}$$
and note that due to our conditions on $p$ we have $\error=o(1)$. Set 
$$t_0:=\left((1+\error)\frac{(r-1)!}{np^r}\right)^{1/(r-1)}.$$
Let $\hat{\pi}(t):=\mathbb{P}[\mathrm{Bin}(t,p)\geq r]$ and define
$$a_c:=-\min_{t\leq t_0} \frac{n\hat{\pi}(t)-t}{1-\hat{\pi}(t)}$$
(which is defined essentially as $a_c^*$ in \cite{MR3025687}, except we replace $\mathrm{Po}(tp)$ with $\mathrm{Bin}(t,p)$). 
In addition denote by $t_c$ the smallest value $t$ where this minimum is reached. Similarly to Lemma 9.5 in \cite{MR3025687} one can show that
$$t_c=(1+o(1))\left(\frac{(r-1)!}{np^r}\right)^{1/(r-1)}$$ 
and 
$$a_c=(1+o(1))\left(1-\frac{1}{r}\right)t_c.$$ 

We begin with the size of the final infected set in the subcritical case, in which the number of the initially infected vertices is smaller than the critical value $a_c$.
\begin{theorem}\label{mainsub}
Let $\tendstoinfty:=\tendstoinfty(n)$ be any function satisfying the conditions $\tendstoinfty=\omega(\sqrt{a_c})$ and $\tendstoinfty\leq a_c-r$. \linebreak[4] 
If $|\infec(0)|=a_c-\tendstoinfty$, then with probability at least
	$$1-\exp\left(-(1+o(1))\frac{r \tendstoinfty^2}{2 (t_0+r\tendstoinfty/3)}\right)$$
	we have $|\infec_f|< t_c$.
\end{theorem}

Following is our main result, which says that in the supercritical case, i.e.\ when the number of initially infected vertices is larger than $a_c$, almost every vertex becomes infected.
\begin{theorem}\label{mainsup}
	
	Let $\tendstoinfty:=\tendstoinfty(n)$ be any function satisfying the conditions $\tendstoinfty=\omega(\sqrt{a_c})$ and $\tendstoinfty\leq t_0-a_c$.\linebreak[4] 
If $|\infec(0)|=a_c+\tendstoinfty$, then with probability at least
$$1-\exp\left(-(1+o(1))\frac{r\tendstoinfty^2}{8(t_0+r\tendstoinfty/3)}\right)-\exp\left(-(1+o(1))\frac{(r-1)\tendstoinfty^2}{8(t_0+(r-1)\tendstoinfty/2)}\right)$$
	we have $|\infec_f|= (1+o(1))n$.
\end{theorem}

\paragraph{Proof Technique.}
When the number of infected vertices is small (at most $t_0$), we need to perform all calculations very carefully, because a small difference in the number of infected vertices can change the outcome of the process significantly. In order to achieve this we introduce a martingale to show that the number of infected vertices is concentrated around its expectation with \emph{exponentially} high probability. The martingale is similar to the one used in \cite{MR3025687}, however the maximal one step difference in our martingale is significantly smaller and thus provides tighter concentration (Lemma \ref{conc}). 

In the subcritical regime, the expected number of infected vertices is less than $t_c\leq t_0$ and therefore the martingale argument alone implies the result (Section \ref{subcritical}).

In the supercritical case, the martingale argument is not sufficient as the probability that the martingale is sufficiently concentrated decreases significantly once $t_0$ vertices have been infected. However it still ensures the number of infected vertices will reach $t_0$ with exponentially high probability. In fact, at least $t_0+\tendstoinfty/2$ vertices become infected (Lemma \ref{early}). Now take a subset of the infected vertices of  size $t_0+\tendstoinfty/4$ and consider the vertices with at least $r-1$ neighbours in that set. The size of this set is roughly $(1+\varepsilon)p^{-1}$ for some positive sequence $\varepsilon:=\varepsilon(n,\tendstoinfty,\error)$ (Lemma \ref{almostinfected1}) and the subgraph spanned by these vertices is a binomial random graph, $G((1+\varepsilon)p^{-1},p)$. 
A closer examination of $\varepsilon$ reveals that with respect to the appearance of the giant component, we are in the weakly supercritical regime and a recent result of Bollob\'as and Riordan \cite{arXiv:1403.6558} (Theorem \ref{lingiant}) implies that this random graph has a giant component of size $2\varepsilon p^{-1}/(1+2\varepsilon)$ with exponentially high probability (Lemma \ref{giant}). Should any vertex in the giant component have an additional infected neighbour, then every vertex in the giant component will become infected eventually. We show that this happens with exponentially high probability. This implies a significant increase in the number of infected vertices and at this point we already have $\omega(t_0)$ infected vertices.
After this, the process ends in two steps, which can be shown by two simple applications of the Chernoff bound (Lemmas \ref{Chernoff1} and \ref {Chernoff2}).

\section{Preliminaries}

We will use the following form of the Chernoff bound.

\begin{theorem}\cite{MR2283885}\label{chernoff}
Let $X\sim\mathrm{Bin}(n,p)$, i.e.\ a binomial random variable with parameters $n$ and $p$. Then for any $\lambda>0$ 
$$\mathbb{P}[X-\mathbb{E}[X]\leq -\lambda]\leq \exp\left(-\frac{\lambda^2}{2\mathbb{E}[X]}\right)$$
and
$$\mathbb{P}[X-\mathbb{E}[X]\geq \lambda]\leq \exp\left(-\frac{\lambda^2}{2(\mathbb{E}[X]+\lambda/3)}\right).$$
\end{theorem}
\smallskip

Let $M_0,\ldots,M_k$ be a sequence of random variables. For each $0\le i\le k$ denote by $\filter{i}$ the $\sigma$-algebra generated by the random variables $M_0,\ldots,M_i$, i.e., $\filter{i}=\sigma(\{M_j\}_{0\leq j\le i})\}$, and we call the sequence $\{\filter{i}\}_{0\leq i \leq k}$ the natural filtration of $M_0,\ldots, M_k$.

The following martingale concentration bound which is a slightly generalised form of a result due to Chung and Lu \cite{MR2283885} (a strengthened version of McDiarmid \cite{MR1678578}) will prove to be vital. For completeness we provide a proof in the appendix.
\begin{theorem}\label{marconc}
For $m_0\in \mathbb{R}$ let $M_0=m_0,M_1,\ldots,M_k$ be a martingale with respect to its natural filtration  $\{\filter{i}=\sigma(\{M_j\}_{0\leq j\le i})\}_{0\leq i \leq k}$ 
whose conditional variance and differences satisfy the following conditions for each $1\leq i \leq k$.
\begin{enumerate}[label=(\roman*)]
\item \label{bounded_diff} $|M_i-M_{i-1}|\leq m$ for some $m\in \mathbb{R_+}$;
\item \label{bounded_var}$\mathrm{Var}[M_i|\filter{i-1}]=\mathbb{E}[(M_i-\mathbb{E}[M_i|\filter{i-1}])^2|\filter{i-1}]\leq \sigma_i^2$ for some $\sigma_i\in \mathbb{R}$.
\end{enumerate}

Then for any $\lambda>0$, we have
$$\mathbb{P}[M_k - m_0 \geq \lambda]\leq \exp\left(-\frac{\lambda^2}{2\left(\sum_{i=1}^k \sigma_i^2+m\lambda/3 \right)}\right).$$
\end{theorem}

We will also need the following result on the appearance of a giant component in $G(n,p)$ by Bollob\'as and Riordan, which follows from Theorem 3 in \cite{arXiv:1403.6558}.

\begin{theorem}\label{lingiant}
Let $\varepsilon=\varepsilon(n)$ satisfy $\varepsilon=O(1)$ and $\varepsilon^3n=\omega(1)$. Denote by $\rho\in(0,1)$ the unique positive solution of $1-\rho=\exp(-(1+\varepsilon)\rho)$. For any $0\leq \gamma < \rho$ with probability $1-\exp(-\Omega(\gamma^2\varepsilon n))$ the binomial random graph $G(n,(1+\varepsilon)/n)$ has a component of size at least $\rho n-\gamma n$.
\end{theorem}

\section{Setup: Martingale}\label{sec:setup}

In order to analyse the bootstrap percolation on $G(n,p)$ we will use the following reformulation similar to Scalia-Tomba \cite{MR3025687,MR798872}. Roughly speaking we examine the infected vertices one by one and determine the vertices which have at least $r$ neighbours in the set of previously examined vertices. The set of examined vertices until step $t$ is denoted by $\checked(t)$ and the set of infected vertices by $\infec(t)$.
 
Formally let $\infec(0)$ be the set of initially infected vertices chosen uniformly at random from the vertex sets of size $a\in [n]$ and without loss of generality we may assume that $\infec(0)=\{1,...,a\}$. Set $\checked(0)=\emptyset$. For each step $t\in\mathbb{N}$, if $\infec(t-1)\backslash \checked(t-1)\neq\emptyset$, then let $U(t)=\{u(t)\}$, where $u(t)$ is the smallest vertex in $\infec(t-1)\backslash \checked(t-1)$, otherwise set $U(t)=\emptyset$. Set $\checked(t):=\checked(t-1)\cup U(t)$. Now for $t\geq 0$ and $i\in[n-a]:=\{1,\ldots, n-a\}$ let $X(t,i)$ be the indicator random variable for the event that the vertex $a+i$ has at least $r$ neighbours in $\checked(t)$ and set $$\infec(t):=\infec(0)\cup\{a+i:X(t,i)=1, i\in [n-a]\}.$$
The process stops when $t=n$. 

Clearly $\checked(t)\subset \infec(t)$. Let $T$ denote the smallest value of $t$ such that $\infec(t)=\checked(t)$. Note that $t\leq T$ implies that $|\checked(t)|=t$ and thus $T$ is also the smallest $t$, in fact the only $t$ satisfying $|\infec(t)|=t$. Since $|\infec(t)|\leq n$ for every $t\in \{0,1,\ldots, n\}$, 
we have that $T\leq n$. Note further that $\infec(T)=\infec_f$. 

In \cite{MR3025687} a martingale was introduced to analyse this process. Note that in any step of the process at most $n-|\infec(t)|$ vertices may become infected (which in the early stages of the process is $(1+o(1))n$) and martingale concentration inequalities depend on the maximal one step difference. Therefore only very weak concentration can be shown for this martingale.

In order to achieve a better control on the maximal one step difference in a martingale,
we refine the process by dividing every step into rounds, in such a way that in each round \emph{exactly one vertex} $v\in [n]\backslash \infec(0)$ is examined (regardless whether it was examined in earlier rounds or not). Thus each step $t\leq n$ consists of $n-a$ rounds and round $i$ of step $t$ is denoted by $(t,i)$. We denote the round following $(t,i)$ by $(t,i)+1$  and the round preceding $(t,i)$  by $(t,i)-1$. Also the ordering of the rounds is given by the lexicographical order, i.e. $(\tau,\iota)<(t,i)$ if either $\tau<t$ or $\tau=t$ and $\iota<i$.

In round $i$ of step $t$ we examine if the vertex $a+i$ has at least $r$ neighbours in $\checked(t)$ and if it does we add it to the set of infected vertices. Formally for each \begin{math}(t,i)\geq (1,1)\end{math} we let
\begin{align}\label{infectedset}
\infec(t,i):=\infec(0)&\cup \{a+j: X(t,j)=1, j=1,\ldots,i\}\cup \{a+j: X(t-1,j)=1, j=i+1,\ldots, n-a\}.
\end{align}
Clearly we have $\infec(t)=\infec(t,n-a)$. For consistency define $\infec(0,n-a):=\infec(0)$.

Define a function $\pi:\mathbb{N}\rightarrow [0,1]$ by
$$\pi(t):=\left\{
\begin{array}{ll}
\mathbb{P}[\mathrm{Bin}(t,p)\geq r],  & \mbox{for } t\leq T;\\ 
\mathbb{P}[\mathrm{Bin}(T,p)\geq r], & \mbox{for } t > T .
\end{array}
\right.$$
Note that $\pi(t)$ is a random variable.

For \begin{math}(t,i)\geq (0,n-a) \end{math}, define the random variable 
\begin{equation}\label{martingaledef}
M(t,i):= \sum_{j=1}^i \frac{X(t,j)-\pi(t)}{1-\pi(t)}+\sum_{j=i+1}^{n-a}\frac{X(t-1,j)-\pi(t-1)}{1-\pi(t-1)}.
\end{equation}
Since $\pi(t)<1$ for every $t\in\mathbb{N}$, the random variable $M(t,i)$ is well defined.
Tedious arguments show that the sequence of random variables $M(0,n-a),\ldots, M(n,n-a)$ forms a martingale (see the appendix for the proof).

\begin{lemma}\label{martingale}
For each $ (0,n-a)\le (t,i)\le (n,n-a)$ we denote by $\filter{t,i}$ the $\sigma$-algebra  generated by $M(0,n-a),\ldots,M(t,i)$. 
The sequence of  $M(0,n-a),\ldots, M(n,n-a)$ forms a martingale with respect to the filtration $\{\filter{0,n-a},\ldots,\filter{n,n-a}\}$.
\end{lemma}

Next we show that $M(t,i)$  is concentrated around its expectation.

\begin{lemma}\label{conc}
Let $t\in \{0,\ldots,n\}$ and $\lambda\in \mathbb{R}_+$ be given. Then we have
$$\mathbb{P}\left[\bigwedge_{(0,n-a)\leq(\tau,i)\leq (t,n-a)} \Big(M(\tau,i)>-\lambda\Big) \right]\geq 1-\exp\left(-\frac{\lambda^2(1-\hat{\pi}(t))^3}{2(n\hat{\pi}(t)+\lambda/3)}\right)$$
and
$$\mathbb{P}\left[\bigwedge_{(0,n-a)\leq(\tau,i)\leq (t,n-a)} \Big(M(\tau,i)<\lambda\Big) \right]\geq 1-\exp\left(-\frac{\lambda^2(1-\hat{\pi}(t))^3}{2(n\hat{\pi}(t)+\lambda/3)}\right).$$
\end{lemma}

\begin{proof}
We  only need to consider the bound on the probability that $M(\tau,i)< \lambda$ for each \linebreak[4]
$(0,n-a)\leq(\tau,i)\leq (t,n-a)$. 
The other case follows simply from the fact that if the sequence of random variables $M(0,n-a),\ldots,M(t,n-a)$ forms a martingale with respect to a filtration, then $-M(0,n-a),\ldots,-M(t,n-a)$ is also a martingale with respect to the same filtration and they both have the same conditional variance and maximal difference.
In order to show that the bounds hold for each step, we introduce the following martingale
$$\hat{M}(\tau,i)=\left\{
\begin{array}{ll}
M(\tau,i) &\mbox{if } \hat{M}((\tau,i)-1)< \lambda, \\
\hat{M}((\tau,i)-1) & \mbox{otherwise}.
\end{array}
\right. $$
Similarly to \begin{math}M(\tau,i)\end{math} we denote  by \begin{math} \hat{\mathcal{F}}(\tau,i) \end{math} the $\sigma$-algebra generated by \begin{math} \hat{M}(0,n-a),\ldots,\hat{M}(\tau,i) \end{math}.
Note that if there exists a round $(\tau,i)$ such that $M(\tau,i)\geq \lambda$, then we have $\hat{M}(\tau',i')\geq \lambda$ for every $(\tau',i')\geq (\tau,i)$. Therefore $\hat{M}(t,n-a)< \lambda$ implies that for every $(0,n-a)\leq(\tau,i)\leq (t,n-a)$ we have $M(\tau,i)< \lambda$. Thus it suffices to show that
\begin{equation}\label{suffcond}
\mathbb{P}\left[\hat{M}(t,n-a)\geq \lambda \right]\leq \exp\left(-\frac{\lambda^2(1-\hat{\pi}(t))^3}{2(n\hat{\pi}(t)+\lambda/3)}\right).
\end{equation}

Our aim is to apply Theorem \ref{marconc} to $\hat{M}(\tau,i)$. For this we need an upper bound on the maximal one step difference and conditional variance.

First we give an upper bound on the one step difference. 
We can easily compute
\begin{equation}\label{onestepdiff}
M(t,i)-M((t,i)-1)\stackrel{\eqref{martingaledef}}{=}\frac{X(t,i)-\pi(t)}{1-\pi(t)}-\frac{X(t-1,i)-\pi(t-1)}{1-\pi(t-1)}.
\end{equation}
By \eqref{onestepdiff} and $\pi(\tau)\leq \hat{\pi}(\tau)$, we have for $\tau\leq t$ 
\begin{align*}
|\hat{M}(\tau,i)-\hat{M}((\tau,i)-1)|&
\leq\max\left\{\frac{\hat{\pi}(\tau)}{1-\hat{\pi}(\tau)}-\frac{\hat{\pi}(\tau-1)}{1-\hat{\pi}(\tau-1)},1+\frac{\hat{\pi}(\tau-1)}{1-\hat{\pi}(\tau-1)}\right\}.
\end{align*}
Since $\hat{\pi}(\tau)=\mathbb{P}[\mathrm{Bin}(\tau,p)\geq r]$ we have
\begin{align*}
\hat{\pi}(\tau)-\hat{\pi}(\tau-1)=\binom{\tau-1}{r-1}p^r(1-p)^{\tau-r}\leq \binom{\tau}{r-1}p^r(1-p)^{\tau-r}\\
\stackrel{p=o(1)}{\leq} \binom{\tau}{r-1}p^{r-1}(1-p)^{\tau-r+1}\leq \sum_{i=0}^{r-1}\binom{\tau}{i}p^i(1-p)^{\tau-i}=1-\hat{\pi}(\tau)
\end{align*}
and consequently
\begin{equation*}
\frac{\hat{\pi}(\tau)}{1-\hat{\pi}(\tau)}-\frac{\hat{\pi}(\tau-1)}{1-\hat{\pi}(\tau-1)}\leq 1+\frac{\hat{\pi}(\tau-1)}{1-\hat{\pi}(\tau-1)}.
\end{equation*}
Therefore we obtain
\begin{align}\label{onestepdiff-bound}
|\hat{M}(\tau,i)-\hat{M}((\tau,i)-1)|\leq 1+\frac{\hat{\pi}(\tau-1)}{1-\hat{\pi}(\tau-1)}=\frac{1}{1-\hat{\pi}(\tau-1)}\stackrel{\tau\leq t}{\leq}\frac{1}{1-\hat{\pi}(t)}.
\end{align}

Next we examine the conditional variance. Conditional on $\hat{M}((\tau,i)-1)\geq \lambda$, the random variable $\mathrm{Var}[\hat{M}(\tau,i)|\hat{\mathcal{F}}((\tau,i)-1)]=0$, on the other hand conditional on $\hat{M}((\tau,i)-1)< \lambda$ we have
$$\mathrm{Var}[\hat{M}(\tau,i)|\hat{\mathcal{F}}((\tau,i)-1),\hat{M}((\tau,i)-1)< \lambda]
= \mathrm{Var}[M(\tau,i)|\filter{(\tau,i)-1}].$$
Note that 
\begin{align*}
\mathrm{Var}[M(\tau,i)|\filter{(\tau,i)-1}]&=\mathrm{Var}\left[\left.\frac{X(\tau,i)}{1-\pi(t)}\right| \filter{(\tau,i)-1}\right] \leq\frac{1}{(1-\hat{\pi}(t))^2}\mathrm{Var}[X(\tau,i)| \filter{(\tau,i)-1}].
\end{align*}
We shall show that
\begin{equation}\label{variance}
\mathrm{Var}[X(\tau,i)| \filter{(\tau,i)-1}]\leq \frac{\hat{\pi}(\tau)-\hat{\pi}(\tau-1)}{1-\hat{\pi}(\tau)}.
\end{equation}
Recall that $X(\tau-1,i)=1$ implies $X(\tau,i)=1$ and that $\tau>T$ implies $X(\tau,i)=X(\tau-1,i)$. 
In both of these cases we have
\begin{align*}
\mathrm{Var}[X(\tau,i)| \filter{(\tau,i)-1},X(\tau-1,i)=1]&=\mathrm{Var}[X(\tau,i)| \filter{(\tau,i)-1},\tau>T]=0\\
&\stackrel{\hat{\pi}(\tau)\geq \hat{\pi}(\tau-1)}{\leq} \frac{\hat{\pi}(\tau)-\hat{\pi}(\tau-1)}{1-\hat{\pi}(\tau)}. 
\end{align*}
Now assume $\tau\leq T$ and $X(\tau-1,i)=0$. Since $X(\tau,i)$ is an indicator random variable, we have
\begin{align*}
\mathrm{Var}[X(\tau,i)| \filter{(\tau,i)-1},\tau\leq T, X(\tau-1,i)=0]&\leq \mathbb{E}[X(\tau,i)| \filter{(\tau,i)-1},\tau\leq T, X(\tau-1,i)=0]\\
&\stackrel{\eqref{condexpX}}{=}\frac{\hat{\pi}(\tau)-\hat{\pi}(\tau-1)}{1-\hat{\pi}(\tau-1)}
\leq\frac{\hat{\pi}(\tau)-\hat{\pi}(\tau-1)}{1-\hat{\pi}(\tau)}
\end{align*}
and thus \eqref{variance} holds.
Therefore we have
\begin{align}\label{condvariance}
\sum_{\tau=1}^{t}\sum_{i=1}^{n-a}\mathrm{Var}[M(\tau,i)|\filter{(\tau,i)-1}] &\stackrel{\eqref{variance}}{\leq} \sum_{\tau=1}^t\sum_{i=1}^{n-a}\frac{\hat{\pi}(\tau)-\hat{\pi}(\tau-1)}{(1-\hat{\pi}(\tau))^3}\nonumber\\
&\leq \sum_{\tau=1}^t \frac{n(\hat{\pi}(\tau)-\hat{\pi}(\tau-1))}{(1-\hat{\pi}(t))^3}\leq \frac{n\hat{\pi}(t)}{(1-\hat{\pi}(t))^3}.
 \end{align}
Note that $M(0,n-a)=0$. Thus by Theorem \ref{marconc} with \eqref{onestepdiff-bound} and \eqref{condvariance} we have
\begin{equation*}
\mathbb{P}[\hat{M}(t,n-a)\geq \lambda] 
\leq \exp\left(-\frac{\lambda^2}{2}\left(\frac{n\hat{\pi}(t)}{(1-\hat{\pi}(t))^3} +\frac{\lambda}{3(1-\hat{\pi}(t))}\right)^{-1}\right)
\end{equation*}
implying \eqref{suffcond}. This completes the proof.
\end{proof}

The previous lemma allows us to analyse the process in the first $t_0$ steps. This will be used in the proofs of Theorems \ref{mainsub} and \ref{mainsup}.

\section{Subcritical case: Proof of Theorem \ref{mainsub}}\label{subcritical}

We want to investigate the number of infected vertices at time $t_c$. By the definition of $a_c$ and $t_c$, we have
\begin{equation}\label{critical}
a_c=-\min_{t\leq t_0}\frac{n\hat{\pi}(t)-t}{1-\hat{\pi}(t)}=\frac{t_c-n\hat{\pi}(t_c)}{1-\hat{\pi}(t_c)}.
\end{equation}

Using \eqref{infectedset} and \eqref{martingaledef} we can express the number of infected vertices until step $t$ by
\begin{equation}\label{marinfec}
|\infec(t)|\stackrel{\eqref{infectedset}}{=}a+\sum_{i=1}^{n-a}X(t,i)\stackrel{\eqref{martingaledef}}{=}a+M(t,n-a)(1-\pi(t))+(n-a)\pi(t)
\end{equation}
and in particular we have for step $t_c$
\begin{align*}
|\infec(t_c)|&=a+(1-\pi(t_c))M(t_c,n-a)+(n-a)\pi(t_c).
\end{align*}
Since $\pi(t)\leq \hat{\pi}(t)$ and $a=a_c-\tendstoinfty$, we obtain
\begin{align}
|\infec(t_c)|&\leq a+M(t_c,n-a)+(n-a)\hat{\pi}(t_c)\nonumber\\
&=(a_c-\tendstoinfty)(1-\hat{\pi}(t_c))+n\hat{\pi}(t_c)+M(t_c,n-a)\nonumber\\
&\stackrel{\eqref{critical}}{=}t_c-n\hat{\pi}(t_c)+n\hat{\pi}(t_c)-\tendstoinfty(1-\hat{\pi}(t_c))+M(t_c,n-a)\nonumber\\
&=t_c-\tendstoinfty(1-\hat{\pi}(t_c))+M(t_c,n-a).\label{numinfver}
\end{align}
Since $t_0=(1+o(1))((r-1)!/(np^r))^{1/(r-1)}$, we have
\begin{equation}\label{t0small}
t_0 p=O\left(\left(\frac{1}{np^r}\right)^{1/(r-1)}p\right)=O\left(\left(\frac{1}{np}\right)^{1/(r-1)}\right)\stackrel{np=\omega(1)}{=}o(1).
\end{equation}
Furthermore,
\begin{align}
\hat{\pi}(t_0)&=\mathbb{P}[\mathrm{Bin}(t_0,p)\geq r]=\sum_{j=r}^{t_0}\binom{t_0}{j}p^{j}(1-p)^{t_0-j}\stackrel{\eqref{t0small}}{=}(1+o(1))\frac{t_0^r p^r}{r!}\nonumber\\
&=(1+o(1))\frac{t_0^{r-1}p^r}{r!}t_0=(1+o(1))\frac{1}{r}\frac{t_0}{n}\stackrel{np=\omega(1)}{=}o\left(t_0 p\right)\stackrel{\eqref{t0small}}{=}o(1).\label{probt0}
\end{align}
Applying Lemma \ref{conc} with $\lambda=(1-\sqrt{\hat{\pi}(t_c)})\tendstoinfty$, we have 
\begin{align}\label{boundonMtc}
M(t_c,n-a)< (1-\sqrt{\hat{\pi}(t_c)})\tendstoinfty
\end{align}
 with probability at least
\begin{align*}
1-\exp\left(-(1+o(1))\frac{\tendstoinfty^2}{2(n\hat{\pi}(t_c)+\tendstoinfty/3)}\right)&\stackrel{\hat{\pi}(t_c)\leq \hat{\pi}(t_0)}{\geq}1-\exp\left(-(1+o(1))\frac{\tendstoinfty^2}{2(n\hat{\pi}(t_0)+\tendstoinfty/3)}\right)\\
&\stackrel{\eqref{probt0}}{=} 1-\exp\left(-(1+o(1))\frac{r \tendstoinfty^2}{2(t_0+r\tendstoinfty/3)}\right).
\end{align*}
Furthermore,  \eqref{numinfver} and \eqref{boundonMtc} imply that with probability at least $1-\exp\left(-(1+o(1))\frac{r \tendstoinfty^2}{2(t_0+r\tendstoinfty/3)}\right)$ we have
$$|\infec(t_c)|\leq t_c-(1-\hat{\pi}(t_c))\tendstoinfty+(1-\sqrt{\hat{\pi}(t_c)})\tendstoinfty=t_c+(\hat{\pi}(t_c)-\sqrt{\hat{\pi}(t_c)})\tendstoinfty\stackrel{\hat{\pi}(t_c)<1}{<}t_c$$
and therefore   $|\infec_f|=T< t_c$, as desired.

\section{Supercritical case: Proof of Theorem \ref{mainsup}}

\begin{figure}[H]
\begin{center}
\begin{tikzpicture}[scale=0.94]
\draw (0,0) ellipse (5 and 4);
\node (X00) at ($(0,0)+(130:5 and 4)$) [circle,fill=black, inner sep=0pt] {};
\node (X01) at ($(0,0)+(-130:5 and 4)$) [circle,fill=black, inner sep=0pt] {};
\node (Z) at ($(0,0)+(-20:5 and 4)$) {};
\node (Y) at ($(0,0)+(-160:5 and 4)$) [circle,fill=black, inner sep=0pt] {};

\draw [name path=X00--X01] (X00)--(X01);
\path [name path=Y--Z] (Y)--(Z);

\path [name intersections={of=X00--X01 and Y--Z}];
\draw (Y)-- (intersection-1);

\node at (-4,-1){$\checked(t_1)$};
\node at (-4,-2){$A$};

\draw (-2,0.5) ellipse (1 and 2);

\node (B) at (-2.5,1.5) [circle,fill=black, inner sep=1pt] {};
\node (Z0) at (-4,2) [circle,fill=black, inner sep=1pt] {};
\node (Z1) at (-4,1.5) [circle,fill=black, inner sep=1pt] {};
\draw (-3.75,1.5) arc (190:166:1);
\node at (-3.6,1.25){\footnotesize {$r-1$}};

\draw (Z0)--(B);
\draw (Z1)--(B);
\node at (-2,2) {$\hat{B}$};

\node (B0) at (-2,1) [circle,fill=black, inner sep=1pt] {};
\node (B1) at (-2.5,0.75) [circle,fill=black, inner sep=1pt] {};
\node (B2) at (-1.5,0.7) [circle,fill=black, inner sep=1pt] {};
\node (B3) at (-2,0.5) [circle,fill=black, inner sep=1pt] {};
\node (B4) at (-2.25,0) [circle,fill=black, inner sep=1pt] {};
\node (B5) at (-1.75,0) [circle,fill=black, inner sep=1pt] {};
\node (B6) at (-1.75,-0.55) [circle,fill=black, inner sep=1pt] {};

\draw (B0)--(B1);
\draw (B0)--(B2);
\draw (B0)--(B3);
\draw(B3)--(B4);
\draw (B3)--(B5);
\draw (B5)--(B6);

\node (A0) at (-3.5,-2.5) [circle,fill=black, inner sep=1pt] {};
\draw (A0)--(B1);
\draw (-2,0.25) ellipse (0.75 and 1.25);

\node at (-2,-0.5){$B$};

\draw (1.5,2) ellipse (2 and 1);
\node (C0) at (0.5,2) [circle,fill=black, inner sep=1pt] {};
\draw (B2)--(C0);
\draw (B5)--(C0);
\draw (B6)--(C0);
\draw (-0.5,1.25*2.55/2.25-0.55) arc (230:207:1);
\node at (-0.4,0.7){\footnotesize{$r$}};

\node (C1) at (1,2) [circle,fill=black, inner sep=1pt] {};
\node (C2) at (1.5,2) [circle,fill=black, inner sep=1pt] {};
\node (C3) at (2,2) [circle,fill=black, inner sep=1pt] {};
\node at (2,2.5) {$C$};

\draw (1.5,-2) ellipse (2 and 1);
\node (D0) at (1,-2) [circle,fill=black, inner sep=1pt] {};
\draw (D0)--(C1);
\draw (D0)--(C2);
\draw (D0)--(C3);
\node at (2, -2) {$D$};
\draw (1,0.5) arc (95:71:1.5);
\node at (1.75,0.4){\footnotesize{$r$}};

\end{tikzpicture}
\caption{Spread of the infection}\label{onlypicture}
\end{center}
\end{figure}
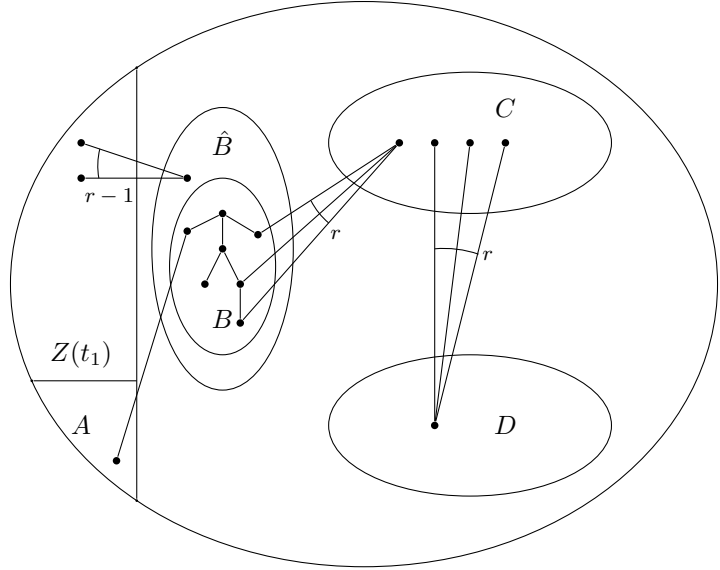

For the remainder of the paper fix an $\tendstoinfty$ satisfying the conditions of Theorem \ref{mainsup}, i.e.,  $\tendstoinfty=\omega(\sqrt{a_c})$ and $\tendstoinfty\leq t_0-a_c$. Define 
$$t_1:=t_0+\tendstoinfty/4.$$

The proof of Theorem \ref{mainsup} consists of several steps (see Figure \ref{onlypicture}). 
First we show that the process runs for at least $t_1$ steps and in fact at this stage there exists a \lq large'  set of vertices  $A\subset \infec(t_1)\backslash \checked(t_1)$ (Lemma~\ref{early}). 
Next we consider the set $\hat{B}$ of vertices which have at least $r-1$ neighbours in $\checked(t_1)$ and prove a lower bound on the number of these vertices (Lemma \ref{almostinfected1}). 
It turns out that the size of $\hat{B}$ is large enough for it to contain a giant component $B$ (Lemma \ref{giant}). 
If there exists a vertex in $A$ which is connected to a vertex in $B$, then every vertex in $B$ will become infected. Finally we examine the set $C$ of vertices which have at least $r$ neighbours in $B$  (Lemma \ref{Chernoff1}) and the set $D$ of vertices which have at least $r$ neighbours in $C$ (Lemma \ref{Chernoff2}). We complete the proof by showing $|D|=(1+o(1))n$, because $D\subset \infec_f$.

\begin{lemma}\label{early}
If $|\infec(0)|=a_c+\tendstoinfty$, then with probability at least
$$1-\exp\left(-(1+o(1))\frac{r \tendstoinfty^2}{8( t_0+r\tendstoinfty/3)}\right)$$
we have $T>t_1$ and $|\infec(t_1)|\geq t_1+(1-2\hat{\pi}(t_0))\tendstoinfty/4$.
\end{lemma}

\begin{proof}
From \eqref{probt0} and Lemma \ref{conc} with $\lambda=\tendstoinfty/2$ we have 
\begin{align}\label{boundMti}
M(t,i)>-\tendstoinfty/2, \quad \forall (0,n-a)\leq(t,i)\leq (t_0,n-a)
\end{align}
with probability  at least
$$1-\exp\left(-(1+o(1))\frac{r \tendstoinfty^2}{8( t_0+r\tendstoinfty/3)}\right).$$

Conditional on this event we shall show that $T>t_0$ and $|\infec(t_0)|\geq t_0+(1-\hat{\pi}(t_0))\tendstoinfty/2$. 
By the definition of $a_c$, for every $t\leq t_0$ we have
\begin{equation}\label{acmin}
a_c\geq \frac{t-n\hat{\pi}(t)}{1-\hat{\pi}(t)}.
\end{equation}

We will indeed show by induction that for every $0\le t\leq t_0$ we have that $|\infec(t)|\geq t+(1-\hat{\pi}(t))\tendstoinfty/2$. Note that this implies that $|\infec(t)|>t$ for every $0\le t\leq t_0$ and thus $T> t_0$. 
Clearly $|\infec(0)|=a_c+\tendstoinfty\geq \tendstoinfty/2$. Now assume that for some $0\le t\leq t_0$ we have for every $\tau< t$ that $|\infec(\tau)|\geq \tau+(1-\hat{\pi}(\tau))\tendstoinfty/2$. This implies that $T>t-1$ and thus $\pi(t)=\hat{\pi}(t)$. 
Therefore we obtain
\begin{align*}
|\infec(t)|&\stackrel{\eqref{marinfec}}{=}a+(1-\pi(t))M(t,n-a)+(n-a)\pi(t)\\
&\stackrel{\eqref{boundMti}}{>} (1-\hat{\pi}(t))(a_c+\tendstoinfty)+n\hat{\pi}(t)-(1-\hat{\pi}(t))\tendstoinfty/2\\
&\stackrel{\eqref{acmin}}{\geq} t+(1-\hat{\pi}(t))\tendstoinfty/2.
\end{align*}

For every $t_1\geq t \geq t_0$ we have $|\infec(t)|\geq |\infec(t_0)|\geq t_0 + (1-\hat{\pi}(t_0))\tendstoinfty/2$, and thus 
$$|\infec(t)|\geq t+(t_0-t_1) + (1-\hat{\pi}(t_0))\tendstoinfty/2= t+ (1-2\hat{\pi}(t_0))\tendstoinfty/4,$$
which implies the result.
\end{proof}

Next we shall establish the size of the giant component in the set of vertices which have at least $r-1$ neighbours in $\checked(t_1)$. For this we first need to establish the total number of vertices which have at least $r-1$ neighbours in $\checked(t_1)$.

\begin{lemma}\label{almostinfected1}
Conditional on $T>t_1$ and $|\infec(t_1)|\geq t_1+(1-2\hat{\pi}(t_0))\tendstoinfty/4$, we have that for any set $A\subset [n]\backslash \checked(t_1)$ satisfying $|A|=(1-2\hat{\pi}(t_0))\tendstoinfty/4$, with probability at least $$1-\exp(-\omega(\tendstoinfty^2/t_0))$$ 
the set of vertices in $[n]\backslash (\checked(t_1)\cup A)$ with at least $r-1$ neighbours in $\checked(t_1)$ is at least
$$\left(1+\frac{3}{4}\error+\frac{(r-1)\tendstoinfty}{4 t_0}\right)p^{-1}.$$ 
\end{lemma}

\begin{proof}
Let $\mathcal{E}$ be the event that the set of vertices in $[n]\backslash (\checked(t_1)\cup A)$ with at least $r-1$ neighbours in $\checked(t_1)$ is at least
$\left(1+\frac{3}{4}\error+\frac{(r-1)\tendstoinfty}{4 t_0}\right)p^{-1}.$ 
Since $A(0)$ contains the smallest vertices in the graph and in every step of the process the smallest vertex from $\infec(t)\backslash\checked(t)$ is selected we have that $\infec(0)=\checked(a)\subset\checked(t_1)$. 
By definition $\checked(t_1)\subset \infec(t_1)$. 
Therefore the result follows once we show that for every $\checked',\infec'$ satisfying $\infec(0)\subset\checked'\subset \infec'$, $|\checked'|=t_1$ and $|\infec'|\geq t_1+(1-2\hat{\pi}(t_0))\tendstoinfty/4$ the probability that $\mathcal{E}$ holds conditional on the event $\checked(t_1)=\checked'$ and $\infec(t_1)=\infec'$  is at least
$$1-\exp(-\omega(\tendstoinfty^2/t_0)).$$

For the remainder of the proof fix $Z'$ and $A'$ which satisfy the previous conditions. 
Define bootstrap percolation restricted to $A'$ as follows. Let $\infec'(0)=\infec(0)$ and $\checked'(0)=\emptyset$. For each step $t\in\mathbb{N}$, if $\infec'(t-1)\backslash \checked'(t-1)\neq\emptyset$, then let $U'(t)=\{u'(t)\}$, where $u'(t)$ is the smallest vertex in $\infec'(t-1)\backslash \checked'(t-1)$, otherwise set $U'(t)=\emptyset$. Set $\checked'(t):=\checked'(t-1)\cup U'(t)$. In addition $\infec'(t)$ contains the set of vertices in $A'$ which have at least $r$ neighbours in $\checked'(t)$ and the set of vertices in $\infec'(0)$. 

Let $\mathcal{G}$ be the event that every vertex outside of $A'$ has less than $r$ neighbours in $Z'$. We will show that the event $\checked(t_1)=Z'\cap\infec(t_1)=A'$ is equivalent to the event $\checked'(t_1)=Z'\cap\infec'(t_1)=A'\cap\mathcal{G}$. 
Note that if $\checked(t_1)=Z'\cap\infec(t_1)=A'$ holds, then any vertex in $[n]\backslash A'$, as it is not contained in $\infec(t_1)=A'$, can have at most $r-1$ neighbours in $Z'=\checked(t_1)$. 
Therefore the event $\checked(t_1)=Z'\cap\infec(t_1)=A'$ is contained in $\mathcal{G}$.

Next we show that if $\mathcal{G}$ holds, then for every $t\leq t_1$ we have that $\checked(t)=\checked'(t)$ and $\infec(t)=\infec'(t)$. 
Clearly this holds for $t=0$ and now assume that it holds until step $t-1$. Recall that $\checked(t)=\checked(t-1)\cup U(t)$ and $\checked'(t)=\checked'(t-1)\cup U'(t)$. 
According to our assumption $\checked(t-1)=\checked'(t-1)$, so we only need to show that $U(t)=U'(t)$.  Recall that $U(t)$ and $U'(t)$ are the smallest vertices in the set $\infec(t-1)\backslash \checked(t-1)$ and $\infec'(t-1)\backslash \checked'(t-1)$ respectively. Because $\infec(t-1)\backslash \checked(t-1)=\infec'(t-1)\backslash \checked'(t-1)$,  we have $U(t)=U'(t)$, which implies $\checked(t)=\checked'(t)$. It only remains to show that $\infec(t)=\infec'(t)$. 
Since $\checked(t)\subset\checked(t_1)=Z'$ and $\mathcal{G}$ holds, we have that any vertex in $\infec(t)$ must be in $A'$, thus $\infec(t)$ is the set of vertices in $A'$ with at least $r$ neighbours in $\checked(t)$ and $\infec(0)$, matching the definition of $\infec'(t)$.

Let $X_v$ be the indicator random variable that a vertex $v\in [n] \backslash (A \cup Z')$ has at least $r-1$ neighbours in $Z'$ and set $X:=\sum_{v\in [n] \backslash (A\cup Z')}X_v$.
For a vertex $v\in A'$ we have
$$\mathbb{P}[X_v=1|\infec(t_1)=A',\checked(t_1)=Z']=1.$$
On the other hand,  when $v\not\in A'$, since the event $\checked(t_1)=Z'\cap\infec(t_1)=A'$ is equivalent to the event $\checked'(t_1)=Z'\cap\infec'(t_1)=A'\cap\mathcal{G}$, we have
\begin{align*}
\mathbb{P}[X_v=1|\infec(t_1)=A',\checked(t_1)=Z']&=\mathbb{P}[X_v=1|\infec'(t_1)=A',\checked'(t_1)=Z',\mathcal{G}].
\end{align*}

Note that $\infec'(t_1)$ and $\checked'(t_1)$ depend only on edges spanned by $A'$, while the events $X_v=1$ and $\mathcal{G}$ depend only on edges with exactly one end in $A'$. 
Therefore the event $\checked'(t_1)=Z'\cap \infec'(t_1)=A'$ is independent of the event $X_v=1\cap\mathcal{G}$ and the event $\mathcal{G}$ as well. 
Therefore
\begin{align*}
\mathbb{P}[X_v=1|\infec'(t_1)=A',\checked'(t_1)=Z',\mathcal{G}]
&=\mathbb{P}[X_v=1|\mathcal{G}].
\end{align*}
Denote by $d_{Z'}(u)$ the number of neighbours of $u$ in $Z'$. We have
$$\mathcal{G}=\left\{\forall u\in [n]\backslash A':d_{Z'}(u)<r\right\}.$$
Note that the event $d_{Z'}(u)$ depends only on edges between $u$ and $Z'$ and the event $X_v=1$ depends only on edges between $v$ and $Z'$. Since these events are independent unless $v=u$, we have
$$\mathbb{P}[X_v=1|\mathcal{G}]=\mathbb{P}[X_v=1|d_{Z'}(v)<r]\geq \mathbb{P}[X_v=1,d_{Z'}(v)<r].$$

In addition, if $v$ has at least $r-1$ neighbours in $Z'$ and $d_{Z'}(v)<r$, then $v$ has exactly $r-1$ neighbours in $Z'$ and thus for $v\not\in A'$ we have
\begin{align*}
\mathbb{P}[X_v=1,d_{Z'}(v)<r]&= \binom{t_1}{r-1}p^{r-1}(1-p)^{t_1-r+1}\\
&=(1+O(t_1^{-1})) \frac{t_1^{r-1}}{(r-1)!}p^{r-1}(1-p)^{t_1-r+1}.
\end{align*}

Since $t_1=O(t_0)$, it follows from \eqref{t0small} that $(1-p)^{t_1-r+1}=1+O(t_0 p)$ and thus

\begin{align*}
\mathbb{P}[X_v=1,d_{Z'}(v)<r]&\geq (1+O(t_1^{-1})+O(t_0p))\frac{t_0^{r-1}(1+(r-1)\tendstoinfty/(4t_0))}{(r-1)!}p^{r-1}\\
&\geq  (1+O(t_1^{-1})+O(t_0p))\left( \frac{1+\error}{np}+\frac{(r-1)\tendstoinfty}{4n p t_0}\right).
\end{align*}

Since $t_1^{-1}=O(t_0^{-1})=o(\error)$, $t_0p=o(\error)$ and $\tendstoinfty=O(t_0)$, we have
$$\mathbb{P}[X_v=1|\infec(t_1)=A',\checked(t_1)=Z']\geq \frac{1+\error+o(\error)}{np}+\frac{(r-1)\tendstoinfty}{4np t_0}.$$

Conditional on $\checked'(t_1)=Z'$, $\infec'(t_1)=A'$ and $\mathcal{G}$, the set of random variables $\{X_v|v\in [n]\backslash A'\}$ are mutually independent. Also conditional on these events we have that $X_v=1$ when $v\in A'\backslash (A\cup Z')$ and thus the set of random variables $\{X_v|v\in [n]\backslash (Z' \cup A)\}$ are also mutually independent. Therefore $X$ stochastically dominates the binomial random variable 
$$\hat{X}=\mathrm{Bin}\left(n-t_1-|A|,\frac{1+\error+o(\error)}{np}+\frac{(r-1)\tendstoinfty}{4np t_0}\right).$$
Since $(t_1+|A|)/n=O(t_1/n)=O((np)^{-r/(r-1)})=o(\error)$, we have
$$\mathbb{E}[\hat{X}]=\frac{1+\error+o(\error)}{p}+\frac{(r-1)\tendstoinfty}{4p t_0}.$$
Recall that $\tendstoinfty\leq t_0$ and thus $\mathbb{E}[\hat{X}]\leq r p^{-1}$.
This and Theorem \ref{chernoff} imply
\begin{align*}
\mathbb{P}\left[X-\mathbb{E}(\hat{X})\leq -\frac{\tendstoinfty}{(t_0p)^{2/3}}\right]
&\leq \exp\left(-\frac{1}{2}\frac{\tendstoinfty^2}{(t_0 p)^{4/3}} \frac{1}{rp^{-1}}\right)
 = \exp\left(-\Omega\left(\frac{\tendstoinfty^2}{t_0 (t_0 p)^{1/3}}\right)\right)
 \stackrel{\eqref{t0small}}{=}\exp\left(-\omega\left(\frac{\tendstoinfty^2}{t_0 }\right)\right).
\end{align*}
We only need to show that
$$\mathbb{E}[\hat{X}]-\frac{\tendstoinfty}{(t_0p)^{2/3}}=\frac{1+\error+o(\error)}{p}+\frac{(r-1)\tendstoinfty}{4pt_0}-\frac{\tendstoinfty}{(t_0p)^{2/3}}\geq \left(1+\frac{3\error}{4}+\frac{(r-1)\tendstoinfty}{4 t_0}\right)p^{-1}.$$
The result follows from the fact that 
$$\frac{\tendstoinfty}{(t_0p)^{2/3}}=\frac{\tendstoinfty}{t_0p}(t_0p)^{1/3}=O((t_0p)^{1/3}p^{-1})=O((np)^{-1/(3(r-1))}p^{-1})=o(\error p^{-1}).$$
\end{proof}

Next we establish the size of a giant component in a set of size $\left(1+\frac{3}{4}\error+\frac{(r-1)\tendstoinfty}{4 t_0}\right)p^{-1}$.

\begin{lemma}\label{giant}
With probability $1-\exp(-\omega(\tendstoinfty^2/t_0))$
the binomial random graph 
$$G\left(\left(1+\frac{3}{4}\error+\frac{(r-1)\tendstoinfty}{4 t_0}\right)p^{-1},p \right)$$
contains a component of size at least
$$\left(\frac{\error}{4}+\frac{(r-1)\tendstoinfty}{2 t_0+(r-1)\tendstoinfty}\right)p^{-1}.$$
\end{lemma}

\begin{proof}
Recall that $\error\geq (np)^{-1/(4(r-1))}$ and thus

\begin{equation}\label{error}
\error^3p^{-1}\geq\left(\frac{1}{np}\right)^{3/(4(r-1))}p^{-1}
=\left(\frac{1}{np^r}\right)^{3/(4(r-1))}p^{-1/4}
=\Omega \left(\left(t_0^{3/4}p^{-1/4}\right)\right)
\stackrel{\eqref{t0small}}{=}\omega(t_0)
=\omega(1).
\end{equation}
Therefore we are in the range where Theorem \ref{lingiant} is applicable. 
Note that for any $\varepsilon>0$ and any \linebreak[4]
$x<(2\varepsilon)/(1+2\varepsilon)$ we have
$$\exp\left((1+\varepsilon)x\right)>\sum_{k=0}^2 \frac{((1+\varepsilon)x)^k}{k!}>1+x+x^2 \frac{1}{1-x}=\frac{1}{1-x}.$$
Thus the unique positive solution $\rho$  of the equation $1-\rho=\exp(-\rho(1+3\error/4+(r-1)\tendstoinfty/(4t_0))$ satisfies 
$$\rho> \frac{6\error/4+2(r-1)\tendstoinfty/(4t_0)}{1+6\error/4+2(r-1)\tendstoinfty/(4t_0)}.$$
Since $\tendstoinfty\leq t_0-a_c=(1+o(1))t_0/r$, we have $(r-1)\tendstoinfty/(4t_0)\leq (r-1)/(4r)+o(1)<1/4$. This together with $\error=o(1)$ implies that
\begin{align}\label{rhodef}
\rho > \frac{6\error/4+2(r-1)\tendstoinfty/(4t_0)}{1+6\error/4+2(r-1)\tendstoinfty/(4t_0)}\geq \frac{\error}{2}+\frac{(r-1)\tendstoinfty}{2(t_0+(r-1)\tendstoinfty/2)}.
\end{align}

Now Theorem \ref{lingiant} implies that with probability 
$$1-\exp(-\Omega(\error^3p^{-1}))\stackrel{\eqref{error}}{=} 1-\exp\left(-\omega(t_0)\right) \stackrel{\tendstoinfty\leq t_0}{=} 1-\exp(-\omega(\tendstoinfty^2/t_0))$$
there is a component of size at least 
$$
\left(\rho-\frac{\delta}{4}\right)p^{-1} 
\stackrel{\eqref{rhodef}}{\geq}  \left(\frac{\error}{4}+\frac{(r-1)\tendstoinfty}{2t_0+(r-1)\tendstoinfty}\right)p^{-1},
$$
completing the proof. 
\end{proof}

In the following two lemmas we estimate the number of vertices with at least $r$ neighbours in a set of size $o(p^{-1})$ and a set of size $\omega(p^{-1})$ separately, because estimating the probability that a vertex has at least $r$ neighbours in such sets differs significantly.
\begin{lemma}\label{Chernoff1}
Let $B,W\subset [n]$ in $G(n,p)$ satisfy $|B|= (np)^{-1/(4(r-1))} p^{-1}/4$ and $|W|=o(n)$. With probability  $1-\exp(-\omega(p^{-1}))$ the number of vertices in $[n]\backslash(B\cup W)$ with at least $r$ neighbours in $B$ is at least $(np)^{1/(4(r-1))}p^{-1}.$
\end{lemma}

\begin{proof}
Let $Y_v$ be the indicator random variable that a vertex $v\in [n] \backslash (B \cup W)$ has at least $r$ neighbours in $B$ and set $Y=\sum_{v\in [n] \backslash (B \cup W)}Y_v$.
Since $|B|p=o(1)$, we have

\begin{align*}
\mathbb{P}[Y_v=1]
&=\sum_{j=r}^{|B|} \binom{|B|}{j}p^{j}(1-p)^{|B|-j}
=(1+o(1)) \binom{|B|}{r}p^r(1-p)^{|B|-r}\\
&=(1+o(1)) \frac{|B|^r}{r!}p^r
=(1+o(1)) \frac{(np)^{-r/(4(r-1))}}{4^r r!}.
\end{align*}

Note that
\begin{equation}\label{gain}
1-\frac{r}{4(r-1)}=\frac{3r-4}{4(r-1)}\stackrel{r\geq 2}{\geq} \frac{1}{2(r-1)}.
\end{equation}

Since $|[n]\backslash (B \cup W)|=(1+o(1))n$, we have
\begin{equation}\label{expectationY}
\mathbb{E}[Y]=(1+o(1))n \frac{(np)^{-r/(4(r-1))}}{4^r r!}\stackrel{\eqref{gain}}{\geq}(1+o(1))\frac{(np)^{1/(2(r-1))}}{4^r r!}p^{-1}.
\end{equation}

Furthermore the set of random variables $\{Y_v|v\in [n]\backslash (B \cup W)\}$ are mutually independent. Therefore, by Theorem \ref{chernoff} we have 
\begin{align*}
\mathbb{P}\left[Y\leq (np)^{1/(4(r-1))}p^{-1}\right]&\leq \exp\left(-\frac{(1+o(1))(\mathbb{E}[Y])^2}{2\mathbb{E}[Y]}\right)\\
&\stackrel{\eqref{expectationY}}{\leq}\exp\left(-\Omega\left((np)^{1/(2(r-1))}p^{-1}\right)\right)
\stackrel{np=\omega(1)}{=} \exp\left(-\omega(p^{-1})\right).
\end{align*}
\end{proof}

\begin{lemma}\label{Chernoff2}
Let $C,W\subset [n]$ in $G(n,p)$ satisfy $|C|= (np)^{1/(4(r-1))}p^{-1}$ and $|W|=o(n)$. Then with probability $1-\exp(-\Omega(p^{-1}))$ the number of vertices in $[n]\backslash (C \cup W)$ with at least $r$ neighbours in $C$ is at least $(1+o(1))n$.
\end{lemma}

\begin{proof}
Let $\indic_v$ be the indicator random variable that a vertex $v\in [n] \backslash (C \cup W)$ has less than $r$ neighbours in $C$ and set $\indic:=\sum_{v\in [n] \backslash (C \cup W)}\indic_v$.
Since $|C|p=\omega(1)$, we have 
\begin{align*}
\mathbb{P}[\indic_v=1]&=\sum_{j=0}^{r-1}\binom{|C|}{j}p^{j}(1-p)^{|C|-j}
\leq \exp(-|C|p) \sum_{j=0}^{r-1} \frac{(|C|p)^{j}}{j!}
\leq (1+o(1)) (|C|p)^{r-1} \exp(-|C|p).
\end{align*}
Because $|C|p=\omega(1)$, we have
\begin{equation}\label{exppoly}
\exp(|C|p)=\omega((|C|p)^{5(r-1)}).
\end{equation}
Since $|[n]\backslash (C \cup W)|\leq n$ and $|C|p= (np)^{1/(4(r-1))}$, we have
$$\mathbb{E}[\indic]\leq  (1+o(1))n(|C|p)^{r-1} \exp(-|C|p) \stackrel{\eqref{exppoly}}{=}o(n (|C|p)^{-4(r-1)})=o\left(n (np)^{-1}\right)=o(p^{-1}).$$

Note that the set of random variables $\{\indic_v|v\in [n]\backslash (C \cup W)\}$ are mutually independent and therefore, by Theorem \ref{chernoff} we have 
$$\mathbb{P}[\indic\geq p^{-1}]\leq \exp\left(-\Omega(p^{-1})\right).$$
Thus all but at most $p^{-1}$ vertices in $[n]\backslash (C\cup W)$ have at least $r$ neighbours in $C$, which is at least $n-|C|-|W|-p^{-1}=(1+o(1))n$.
\end{proof}

\begin{proof}[Proof of Theorem \ref{mainsup}]
According to Lemma \ref{early}, with probability at least 
$$1-\exp\left(-(1+o(1))\frac{r \tendstoinfty^2}{8( t_0+r\tendstoinfty/3)}\right)$$
 we have that $T>t_1$ and $|\infec(t_1)|\geq t_1+(1-2\hat{\pi}(t_0))\tendstoinfty/4$. 
Note that $T>t_1$ implies $|\checked(t_1)|=t_1$.

Let $A\subset \infec(t_1)\backslash \checked(t_1)$ with $|A|=(1-2\hat{\pi}(t_0))\tendstoinfty/4  \stackrel{\eqref{probt0} }{=}(1+o(1))\tendstoinfty/4$. Lemma \ref{almostinfected1} implies that conditional on $T>t_1$ and $|\infec(t_1)|\geq t_1+(1-2\hat{\pi}(t_0))\tendstoinfty/4$, with probability at least $1-\exp(-\omega(\tendstoinfty^2/t_0))$ there is a set of vertices in $[n]\backslash (\checked(t_1)\cup A)$ of size at least 
$\left(1+\frac{3}{4}\error+\frac{(r-1)\tendstoinfty}{4 t_0}\right)p^{-1},$ 
 where every vertex in the set has at least $r-1$ neighbours in $\checked(t_1)$. Select a subset $\hat{B}$ of these vertices of size exactly $$\left(1+\frac{3}{4}\error+\frac{(r-1)\tendstoinfty}{4 t_0}\right)p^{-1}.$$ 
Note that until this point every event depends only on edges with one end in $\checked(t_1)$.

Since the graph spanned by $\hat{B}$ is a binomial random graph $G((1+3\error/4+(r-1)\tendstoinfty/(4t_0))p^{-1},p)$, Lemma \ref{giant} implies that with probability $1-\exp(-\omega(\tendstoinfty^2/t_0))$ it contains a component of size at least 
$(\error/4+(r-1)\tendstoinfty/(2 t_0+(r-1)\tendstoinfty)))p^{-1}$.
Note that this event depends only on the edges with both ends in $\hat{B}$ and therefore it is independent of the previous events. Let $B'\subset \hat{B}$ be a component of size exactly 
$$\left(\frac{\error}{4}+\frac{(r-1)\tendstoinfty}{2  t_0+(r-1)\tendstoinfty}\right)p^{-1}.$$

Observe that if a vertex in $A$ is connected to a vertex in $B'$, then every vertex in $B'$ will become infected eventually. 
The probability that no vertex in $A$ is connected to any vertex in $B'$ is 
$$(1-p)^{|A| \cdot |B'| }\leq (1-p)^{(1+o(1))|B'|\tendstoinfty/4}\leq \exp\left(-(1+o(1))\frac{(r-1)\tendstoinfty^2}{8t_0+4(r-1)\tendstoinfty}\right).$$
This event depends on edges between $A$ and $B'$ and thus it is independent of the previous events.

Since $|B'|\geq \error p^{-1}/4\geq (np)^{-1/4(r-1)}p^{-1}/4$, there exists a set $B\subset B'$ with $|B|=(np)^{-1/4(r-1)}p^{-1}/4$. Lemma \ref{Chernoff1} implies that the number of vertices in $[n]\backslash (\checked(t_1)\cup A \cup \hat{B})$ which have at least $r$ neighbours in $B$ is at least $(np)^{1/(4(r-1))}p^{-1}$ and let $C$ be a subset of these vertices with exactly $(np)^{1/(4(r-1))}p^{-1}$ vertices. This event depends only on edges between $B$ and $[n]\backslash (\checked(t_1)\cup A \cup \hat{B})$ and thus it is independent of the previous events.

Finally let $D$ be the set of vertices in $[n]\backslash (\checked(t_1)\cup A \cup B \cup C)$ which contain at least $r$ neighbours in $C$. Therefore, by Lemma \ref{Chernoff2} we have that $|D|=(1+o(1))n$ with probability $1-\exp(-\Omega(p^{-1}))$. Similarly as before this event depends only on edges which haven't been considered previously and thus it is independent of the previous events.

Since $D\subset \infec_f$ and $p^{-1}=\omega(t_0)$, 
we have that the probability that almost every vertex becomes infected is at least
$$1-\exp\left(-(1+o(1))\frac{r\tendstoinfty^2}{8(t_0+r\tendstoinfty/3)}\right)-\exp\left(-(1+o(1))\frac{(r-1)\tendstoinfty^2}{8 t_0+4(r-1)\tendstoinfty}\right),$$
completing the proof.
\end{proof}

\section{Discussion}

In a earlier weaker version of this paper \cite{aofa2016} we consider the case when $\error$ is a constant, indeed $\error=r-1$. The proof for that case is simpler but the exponential tail bounds are weaker, in particular the constant in the exponential is weaker by a factor of roughly $r-1$.

\bibliographystyle{plain}
\bibliography{referencegnpexp}

\newpage
\section{Appendix}

Since throughout the appendix we work with martingales, we start with the well-known definition of a martingale.
\begin{definition}
Let $(\Omega,\mathcal{F},P)$ be a probability space and $\emptyset \subseteq \filter{0} \subseteq \ldots \subseteq \filter{k} \subseteq \mathcal{F}$ be a sequence of sub-$\sigma$-algebras. A sequence of random variables $M_0,\ldots,M_k$ forms a martingale with respect to the filtration $\{\filter{i}\}_{0\leq i \leq k}$ if the following three conditions are satisfied:
\begin{enumerate}[label=(M\arabic*)]
	\item \label{measurable} for every $0\leq i \leq k$, the random variable $M_i$ is $\filter{i}$-measurable;
	\item \label{finite_exp} for every $0\leq i \leq k$, $\mathbb{E}[|M_i|]\leq \infty$;
	\item \label{martingale_cond} for every $1\leq i \leq k$, $\mathbb{E}[M_i|\filter{i-1}]=M_{i-1}$.
\end{enumerate}
\end{definition}

\subsection{Proof of Lemma \ref{martingale}}

In order to prove Lemma \ref{martingale} we let $X(0,1),\ldots,X(n,n-a)$ and $M(0,n-a),\ldots,M(n,n-a)$ be random variables defined in Section \ref{sec:setup}. Recall that for $(t,i)\geq (0,n-a)$ we denote the $\sigma$-algebra  generated by $M(0,n-a),\ldots,M(t,i)$  by $\filter{t,i}=\sigma\left(\left\{M(\tau,\iota)\right\}_{(0,n-a)\leq(\tau,\iota)\leq(t,i)}\right)$. Similarly we denote  the $\sigma$-algebra generated by  $X(0,1),\ldots,X(t,i)$ by $\mathcal{G}(t,i)=\sigma\left(\left\{X(\tau,\iota)\right\}_{(0,1)\leq(\tau,\iota)\leq(t,i)}\right)$.

\begin{lemma}\label{filtrationsmatch}
For any $(0,n-a)\leq (t,i)\leq (n,n-a)$,   $\mathcal{G}(t,i)=\filter{t,i}$.
\end{lemma}

\begin{proof}

The proof is by induction. Since $X(0,1)=\ldots=X(0,n-a)=0$ and $M(0,n-a)=0$ we have that $\filter{0,n-a}=\mathcal{G}(0,n-a)$.

Now assume that $\filter{(t,i)-1}=\mathcal{G}((t,i)-1)$. Since $\filter{(t,i)-1}\subseteq \filter{t,i}$ and $\mathcal{G}((t,i)-1)\subseteq\mathcal{G}(t,i)$, the result follows if we can show that $X(t,i)$ is $\filter{t,i}$-measurable and $M(t,i)$ is $\mathcal{G}(t,i)$-measurable.

First we show that $\pi(t)$ is $\mathcal{G}((t,i)-1)$-measurable, and thus by our induction hypothesis it is also $\filter{(t,i)-1}$-measurable. Recall that
$$|\infec(\tau)|\stackrel{\eqref{infectedset}}{=} a+\sum_{i=1}^{n-a}X(\tau,i)$$
and that $T$ is the only $\tau$ such that $|\infec(\tau)|=\tau$. If $\tau<t$, then the random variables $X(\tau,1),\ldots,X(\tau,n-a)$ are $\mathcal{G}((t,i)-1)$-measurable and thus $\mathds{1}_{T=\tau}$ is as well.
In addition the random variables $\mathds{1}_{T< t}=\sum_{\tau=0}^{t-1} \mathds{1}_{T=\tau}$ and $T \mathds{1}_{T<t}=\sum_{\tau=0}^{t-1}\tau \mathds{1}_{T=\tau}$ are also $\mathcal{G}((t,i)-1)$-measurable.

One can express $\pi(t)=\hat{\pi}(t)(1-\mathds{1}_{T<t})+\hat{\pi}(T \mathds{1}_{T<t})\mathds{1}_{T<t}$ and thus $\pi(t)$ is also $\mathcal{G}((t,i)-1)$-measurable. 

Recall from \eqref{onestepdiff} that
\begin{equation}\label{onestepdiffappendix}
M(t,i)-M((t,i)-1)=\frac{X(t,i)-\pi(t)}{1-\pi(t)}-\frac{X(t-1,i)-\pi(t-1)}{1-\pi(t-1)}.
\end{equation}

The result follows from the fact that every random variable except $M(t,i)$ in \eqref{onestepdiffappendix} is $\mathcal{G}(t,i)$-measurable and every variable except $X(t,i)$ in \eqref{onestepdiffappendix} is $\filter{t,i}$-measurable.
\end{proof}

\begin{proof}[Proof of Lemma \ref{martingale}]
Clearly the conditions \ref{measurable} and \ref{finite_exp} are satisfied, so we only need to show \ref{martingale_cond}. Fix \begin{math}1\leq t\leq n\end{math} and \begin{math}1\leq i \leq n-a\end{math}.

First we shall show that
\begin{align}
\mathbb{E}\left[\left. \frac{X(t,i)-\pi(t)}{1-\pi(t)}\right| \filter{(t,i)-1}\right]
&=\frac{X(t-1,i)-\pi(t-1)}{1-\pi(t-1)}.\label{expectation}
\end{align}

To this end, observe that if $X(t-1,i)=1$, then we have $X(t,i)=1$ and conditional on this event both sides of Equation \eqref{expectation} equal 1. 

Now assume that $X(t-1,i)=0$. When $t>T$, we have $X(t,i)=X(t,i-1)=0$ with probability 1 and by the definition of $\pi(t)$ we have $\pi(t)=\pi(t-1)=\hat{\pi}(T)$. Evaluating both sides of Equation \eqref{expectation} gives us  $-\hat{\pi}(T)/(1-\hat{\pi}(T))$. 

Finally we look at the case when $t\leq T$ and $X(t-1,i)=0$. Since $t\leq T$ we have $\pi(t)=\hat{\pi}(t)$ and thus
\begin{align}
\mathbb{E}&\left[\left.\frac{X(t,i)-\pi(t)}{1-\pi(t)}\right| \filter{(t,i)-1},X(t-1,i)=0,t\leq T\right]\nonumber\nonumber\\
&=\frac{\mathbb{E}[X(t,i)| \filter{(t,i)-1},X(t-1,i)=0,t\leq T]-\hat{\pi}(t)}{1-\hat{\pi}(t)}.\label{eqvcondexp}
\end{align}

Note that $X(t-1,i)=0$ implies $a+i\not\in\infec(t-1)$ and thus $a+i\not\in\checked(t)$. When $t\leq T$, the random variable $X(t,i)$ depends only on edges between $a+i$ and $\checked(t)$ and if $a+i\in V\backslash \checked(t)$ these edges are disjoint from the edges between $j$ and $\checked(t)$ for any $j\neq i$. Therefore, conditional on $t\leq T$ we have that $X(t,i)$ is independent of $X(\tau,\iota)$ for every $(0,n-a)\leq (\tau,\iota)<(t,i)$ such that $\iota\neq i$. This together with Lemma \ref{filtrationsmatch} implies
\begin{align*}
\mathbb{E}\left[X(t,i)| \filter{(t,i)-1},X(t-1,i)=0,t\leq T\right]&=\mathbb{E}\left[X(t,i)| \sigma(\{X(\tau,i)\}_{\tau<t}),X(t-1,i)=0,t\leq T\right].
\end{align*}
Since $X(t-1,i)=0$ implies $X(\tau,i)=0$ for every $\tau<t$, we have
$$\mathbb{E}\left[X(t,i)| \sigma(\{X(\tau,i)\}_{\tau<t}),X(t-1,i)=0,t\leq T\right]=\mathbb{E}[X(t,i)|X(t-1,i)=0,t\leq T].$$
Note that
\begin{align*}
\mathbb{P}\left[X(t,i)=0|X(t-1,i)=0,t\leq T \right]
=\frac{1-\hat{\pi}(t)}{1-\hat{\pi}(t-1)}
=1-\frac{\hat{\pi}(t)-\hat{\pi}(t-1)}{1-\hat{\pi}(t-1)}.
\end{align*}
Therefore we obtain
\begin{equation}\label{condexpX}
\mathbb{E}\left[X(t,i)| X(t-1,i)=0,t\leq T\right]=\frac{\hat{\pi}(t)-\hat{\pi}(t-1)}{1-\hat{\pi}(t-1)},
\end{equation}
which together with \eqref{eqvcondexp} gives
\begin{align*}
\mathbb{E}\left[\left.\frac{X(t,i)-\pi(t)}{1-\pi(t)}\right| \filter{(t,i)-1},X(t-1,i)=0,t\leq T\right]&=-\frac{\hat{\pi}(t)}{1-\hat{\pi}(t)}\cdot \frac{1-\hat{\pi}(t)}{1-\hat{\pi}(t-1)}+1\cdot \frac{\hat{\pi}(t)-\hat{\pi}(t-1)}{1-\hat{\pi}(t-1)}\\
&=-\frac{\hat{\pi}(t-1)}{1-\hat{\pi}(t-1)},
\end{align*}
implying \eqref{expectation}.
Finally  using \eqref{onestepdiffappendix} and  \eqref{expectation} we have  
\begin{align*}
\mathbb{E}&[M(t,i)-M((t,i)-1)|\filter{(t,i)-1}]\\
&\stackrel{\eqref{onestepdiffappendix}}{=}\mathbb{E}\left[\left.\frac{X(t,i)-\pi(t)}{1-\pi(t)}\right| \filter{(t,i)-1}\right]-\frac{X(t-1)-\pi(t-1)}{1-\pi(t-1)}
\stackrel{\eqref{expectation}}{=}0.
\end{align*}
\end{proof}

\subsection{Proof of Theorem \ref{marconc}}

We  need the following well-known results from conditional expectation \cite{MR1155402}.
Let $X,Y$ be random variables with $\mathbb{E}[|X|],\mathbb{E}[|Y|]<\infty$ and let $\mathcal{G}$ be a sub-$\sigma$-algebra of $\mathcal{F}$. Then the following holds.
\begin{enumerate}[label=(C\arabic*)]
	\item \label{exp_of_exp} $\mathbb{E}[\mathbb{E}[X|\mathcal{G}]]=\mathbb{E}[X]$;
	\item \label{linearity} $\mathbb{E}[aX+bY|\mathcal{G}]=a\mathbb{E}[X|\mathcal{G}]+b\mathbb{E}[Y|{\mathcal{G}}]$ for $a,b\in\mathbb{R}$;
	\item \label{known} If $Y$ is $\mathcal{G}$-measurable, then $\mathbb{E}[Y|\mathcal{G}]=Y$;
	\item \label{known_part} If $X,Y>0$, $\mathbb{E}[XY]<\infty$, and $Y$ is $\mathcal{G}$-measurable, then $\mathbb{E}[XY|\mathcal{G}]=Y\mathbb{E}[X|\mathcal{G}]$.
\end{enumerate}

Note that if $\mathcal{G}=\emptyset$, then $\mathbb{E}[X|\mathcal{G}]=\mathbb{E}[X]$. We also need the following result of McDiarmid \cite{MR1678578}.

\begin{lemma}\label{gprop}
Let $g(x)=2\sum_{\ell=2}^{\infty}x^{\ell-2}/\ell!$. Then $g(x)$ is monotone increasing for $x\in \mathbb R_+$ and in addition $g(x)< (1-x/3)^{-1}$ for $x<3$.
\end{lemma}

\begin{proof}[Proof of Theorem ~\ref{marconc}]
Note that for every $1\leq i \leq k$, if $\sigma_i^2=0$, then  $\mathrm{Var}[M_i|\filter{i-1}]=0$. Therefore, if $\sigma_i^2=0$ for every $1\leq i \leq k$, then we have  $M_k=M_{k-1}=M_0=m_0$. So for the remainder of the proof we may assume that $\sum_{i=1}^{k}\sigma_i^2>0$.

For $1\leq i \leq k$ and $t>0$ we have
\begin{equation}\label{expansion}
\mathbb{E}\left[e^{t(M_i-M_{i-1})}|\filter{i-1}\right]=\mathbb{E}\left[\sum_{\ell=0}^{\infty}\frac{t^{\ell}}{\ell!}(M_i-M_{i-1})^{\ell}|\filter{i-1}\right].
\end{equation}
Note that
$$\mathbb{E}[M_i-M_{i-1}|\filter{i-1}]\stackrel{\ref{linearity}}{=}\mathbb{E}[M_i|\filter{i-1}]-\mathbb{E}[M_{i-1}|\filter{i-1}]\stackrel{\ref{known},\ref{measurable}}{=}\mathbb{E}[M_i|\filter{i-1}]-M_{i-1}\stackrel{\ref{martingale_cond}}{=}0.$$
This together with \eqref{expansion} and \ref{linearity} implies
\begin{align*}
\mathbb{E}\left[e^{t(M_i-M_{i-1})}|\filter{i-1}\right]&=1+\mathbb{E}\left[\sum_{\ell=2}^{\infty}\frac{t^{\ell}}{\ell!}(M_i-M_{i-1})^{\ell}|\filter{i-1}\right]\\
&=1+\mathbb{E}\left[\frac{t^{2}}{2}(M_i-M_{i-1})^{2}g(t(M_i-M_{i-1}))|\filter{i-1}\right].
\end{align*}
By the condition $\ref{bounded_diff}$ of Theorem ~\ref{marconc} we have $M_i-M_{i-1}\leq m$ and by Lemma \ref{gprop} the function $g$ is monotone increasing and thus $g(t(M_i-M_{i-1}))\leq g(tm)$, leading to
\begin{align}
\mathbb{E}\left[e^{t(M_i-M_{i-1})}|\filter{i-1}\right]&\leq1+\mathbb{E}\left[\frac{t^{2}}{2}(M_i-M_{i-1})^{2}g(tm)|\filter{i-1}\right]\nonumber \\
&\stackrel{\ref{linearity}}{=}1+\frac{t^{2}}{2}g(tm)\mathbb{E}\left[(M_i-M_{i-1})^{2}|\filter{i-1}\right]\nonumber \\
&\stackrel{\ref{bounded_var}}{\leq}1+\frac{t^{2}}{2}g(tm)\sigma_i^2 \nonumber \\
&\leq \exp\left(\frac{t^{2}}{2}g(tm)\sigma_i^2\right). \label{exp_exponential_diff}
\end{align}

Clearly
$$\mathbb{E}\left[e^{tM_i}|\filter{i-1}\right]=\mathbb{E}\left[e^{t(M_i-M_{i-1})}e^{tM_{i-1}}|\filter{i-1}\right].$$
Note that both $e^{t(M_i-M_{i-1})}>0$ and $e^{tM_{i-1}}>0$. In addition $\ref{bounded_diff}$ implies $M_i\leq m_0+im$ for any $0\leq i \leq k$. Therefore $e^{tM_i}$ is also bounded from above and thus $\mathbb{E}[e^{tM_i}]<\infty$. By \ref{measurable} and \ref{known_part} we have
\begin{equation}\label{exp_exponential}
\mathbb{E}\left[e^{tM_i}|\filter{i-1}\right]=e^{tM_{i-1}}\mathbb{E}\left[e^{t(M_i-M_{i-1})}|\filter{i-1}\right]\stackrel{\eqref{exp_exponential_diff}}{\leq} e^{tM_{i-1}}\exp\left(\frac{t^{2}}{2}g(tm)\sigma_i^2 \right).
\end{equation}
Now we will show by induction that for each $0\leq i \leq k$
$$\mathbb{E}\left[e^{tM_i}\right]\leq \exp\left(tm_0+\frac{t^{2}}{2}g(tm)\sum_{j=1}^i\sigma_j^2 \right).$$
Observe that this holds for $i=0$. Assume that it holds for $i-1$:
\begin{equation}\label{indhyp}
\mathbb{E}[e^{tM_{i-1}}]\leq \exp\left(tm_0+\frac{t^{2}}{2}g(tm)\sum_{j=1}^{i-1}\sigma_j^2 \right).
\end{equation}
Then we have
\begin{align*}
\mathbb{E}\left[e^{tM_{i}}\right]&\stackrel{\ref{exp_of_exp}}{=}\mathbb{E}\left[\mathbb{E}\left[e^{tM_{i}}|\filter{i-1}\right]\right]\\
&\stackrel{\eqref{exp_exponential}}{\leq} \mathbb{E}\left[e^{tM_{i-1}}\exp\left(\frac{t^{2}}{2}g(tm)\sigma_i^2 \right)\right]\\
&\stackrel{\ref{linearity}}{=}\exp\left(\frac{t^{2}}{2}g(tm)\sigma_i^2 \right) \mathbb{E}[e^{tM_{i-1}}]\\
&\stackrel{\eqref{indhyp}}{\leq}\exp\left(tm_0+\frac{t^{2}}{2}g(tm)\sum_{j=1}^i\sigma_j^2 \right).
\end{align*}
In particular, we have
\begin{equation}\label{exp_final}
\mathbb{E}[e^{t(M_k-m_0)}]\leq \exp\left(\frac{t^{2}}{2}g(tm)\sum_{i=1}^k \sigma_i^2 \right).
\end{equation}
By Markov's inequality we have that for any $\lambda,t>0$
\begin{align}
\mathbb{P}[M_k-m_0\geq \lambda]&=\mathbb{P}[e^{t(M_k-m_0)}\geq e^{t\lambda}]\nonumber\\
&\leq e^{-t\lambda}\mathbb{E}[e^{t(M_k-m_0)}]\nonumber\\
&\stackrel{\eqref{exp_final}}{\leq} \exp\left(-t\lambda+\frac{t^{2}}{2}g(tm)\sum_{i=1}^k \sigma_i^2\right). \label{expbound}
\end{align}
Set
\begin{equation}\label{paramchoice}
t=\frac{\lambda}{\sum_{i=1}^k \sigma_i^2+m\lambda/3}.
\end{equation}
Recall that $\sum_{i=1}^{k}\sigma_i^2>0$ and thus
$$tm=\frac{m\lambda}{\sum_{i=1}^k \sigma_i^2+m\lambda/3}<\frac{m\lambda}{m\lambda/3}=3.$$
Therefore by Lemma \ref{gprop} we have $g(tm)\leq (1-tm/3)^{-1}$ and 
\begin{align*}
\mathbb{P}[M_k-m_0\geq \lambda]&\stackrel{\eqref{expbound}}{\leq} \exp\left(\frac{t^{2}}{2}\frac{1}{1-tm/3}\sum_{j=1}^k\sigma_j^2-t\lambda\right)\\
&\stackrel{\eqref{paramchoice}}{=}\exp\left(-\frac{\lambda^2}{2\left(\sum_{i=1}^k \sigma_i^2+m\lambda/3\right)}\right).
\end{align*}

\end{proof}

\end{document}